\documentclass[12pt]{amsart}
\usepackage{amsfonts}
\usepackage{amscd}
\usepackage{amsmath,amssymb,amsthm}
\usepackage{color}
\usepackage{graphicx}
\usepackage[colorlinks=true,linkcolor=red,citecolor=blue]{hyperref}
\usepackage{subfigure,array,tabularx}

\newtheorem{definition}{Definition}[section]
\newtheorem{remark}{Remark}[section]

\newtheorem{theorem}{Theorem}[section]
\newtheorem{lemma}{Lemma}[section]
\newtheorem{corollary}{Corollary}[section]
\numberwithin{equation}{section}

\numberwithin{equation}{section}

\input{tcilatex}

\begin{document}
\title[Obata theorem for the weighted Kohn Laplacian]{On the Obata theorem
for the weighted Kohn Laplacian in a closed weighted Sasakian manifold}
\author{$^{\ast }$Chin-Tung Wu}
\address{Department of applied Mathematics, National Pingtung University,
Pingtung 90003, Taiwan}
\email{ctwu@mail.nptu.edu.tw}
\thanks{$^{\ast }$Research supported in part by the MOST of Taiwan}

\begin{abstract}
In this paper, we generalize the CR Obata theorem for the Kohn Laplacian to
a closed strictly pseudoconvex CR manifold with a weighted volume measure.
More precisely, we first derive the weighted CR Reilly's formula associated
with the weighted Kohn Laplacian and obtain the corresponding first
eigenvalue estimate. With its application, we obtain the CR\ Obata theorem
in a closed weighted Sasakian manifold.
\end{abstract}

\maketitle

\section{Introduction}

Let $(M,J,\theta ,d\sigma )$ be a strictly pseudoconvex CR $(2n+1)$-manifold
with a weighted volume measure $d\sigma =e^{-\phi (x)}\theta \wedge \left(
d\theta \right) ^{n}$ for a given real smooth weighted function $\phi $. In 
\cite{ccw2}, Chang-Chen-Wu derive the weighted CR Reilly formula associated
with the Witten sub-Laplacian $\Delta _{b}-\langle \nabla _{b}\cdot ,\nabla
_{b}\phi \rangle $ in a compact weighted strictly pseudoconvex CR $(2n+1)$%
-manifold, here $\Delta _{b}$ is the sub-Laplacian and $\nabla _{b}$ is the
subgradient. Then they can obtain the first eigenvalue estimate for the
Witten sub-Laplacian and the corresponding CR\ Obata theorem in a compact
weighted Sasakian manifold.

In this paper, we generalize the result to the weighted Kohn Laplacian. More
precisely, we first derive the weighted CR Reilly formula associated with
the weighted Kohn Laplacian%
\begin{equation*}
\begin{array}{c}
\square _{b}^{\phi }=\square _{b}+2\langle \overline{\partial }_{b}\cdot ,%
\overline{\partial }_{b}\phi \rangle%
\end{array}%
\end{equation*}%
on a closed strictly pseudoconvex CR $(2n+1)$-manifold. Here $\square _{b}$
is the Kohn Laplacian and $\overline{\partial }_{b}$ is the conjugate of the
Cauchy-Riemann operator $\partial _{b}$ as in section $2$. Secondly, we
obtain the corresponding first eigenvalue estimate for the weighted Kohn
Laplacian on a closed strictly pseudoconvex CR $(2n+1)$-manifold. With its
application, we obtain the CR Obata theorem in a closed weighted Sasakian
manifold which is served as a generalization of results in \cite{wu} and 
\cite{lsw}. Note that the weighted Kohn Laplacian $\square _{b}^{\phi }$
satisfies the integration by parts equation%
\begin{equation*}
\begin{array}{c}
\int_{M}\langle \square _{b}^{\phi }f,g\rangle d\sigma =2\int_{M}\langle 
\overline{\partial }_{b}f,\overline{\partial }_{b}g\rangle d\sigma
=\int_{M}\langle f,\square _{b}^{\phi }g\rangle d\sigma ,%
\end{array}%
\end{equation*}%
for all smooth functions $f,g$ on $M$.

As in \cite{cckl} the ($\infty $-dimensional) Bakry-Emery pseudohermitian
Ricci curvature $Ric(\square _{b}^{\phi })$ and the corresponding torsion $%
Tor(\square _{b}^{\phi })$ are defined by 
\begin{equation}
\begin{array}{l}
Ric(\square _{b}^{\phi })(W,W)=R_{\alpha \overline{\beta }}W^{\alpha }W^{%
\overline{\beta }}+2\func{Re}(\phi _{\alpha \overline{\beta }}W^{\alpha }W^{%
\overline{\beta }}), \\ 
Tor(\square _{b}^{\phi })(W,W)=2\limfunc{Re}(\phi _{\overline{\alpha }%
\overline{\beta }}W^{\overline{\alpha }}W^{\overline{\beta }}),%
\end{array}
\label{A}
\end{equation}%
for all $W=W^{\alpha }Z_{\alpha }+W^{\overline{\alpha }}Z_{\overline{\alpha }%
}\in T^{1,0}(M)\oplus T^{0,1}(M).$

Now we recall the weighted CR Paneitz operator $P_{0}^{\phi }$ in Section $3$
as%
\begin{equation*}
\begin{array}{c}
P_{0}^{\phi }f=e^{\phi }[\delta _{b}(e^{-\phi }P^{\phi }f)+\overline{\delta }%
_{b}(e^{-\phi }\overline{P}^{\phi }f)],%
\end{array}%
\end{equation*}%
here $P^{\phi }f=\sum_{\beta =1}^{n}(P_{\beta }^{\phi }f)\theta ^{\beta }$
with 
\begin{equation*}
\begin{array}{c}
P_{\beta }^{\phi }f=P_{\beta }f-\partial _{Z_{\beta }}\left\langle \partial
_{b}f,\partial _{b}\phi \right\rangle +n\sqrt{-1}f_{0}\phi _{\beta }%
\end{array}%
\end{equation*}
and $\overline{P}^{\phi }f=\sum_{\beta =1}^{n}(\overline{P}_{\beta }^{\phi
}f)\theta ^{\overline{\beta }}$, the conjugate of $P^{\phi }$. In Lemma \ref%
{lemma}, we observe that%
\begin{equation*}
\begin{array}{l}
P_{0}^{\phi }f=\frac{1}{4}(\square _{b}^{\phi }\square _{b}^{\phi }+%
\overline{\square }_{b}^{\phi }\overline{\square }_{b}^{\phi
})f+n^{2}(f_{00}-\phi _{0}f_{0})-n\func{Re}Q^{\phi }f,%
\end{array}%
\end{equation*}%
for the second-order operator $Q^{\phi }f=2\sqrt{-1}e^{\phi }(e^{-\phi
}A^{\alpha \beta }f_{\alpha })_{,\beta }.$ This implies that $P_{0}^{\phi }$
is a real fourth-order self-adjoint operator on a closed weighted strictly
pseudoconvex CR $(2n+1)$-manifold.

\begin{definition}
Let $(M,J,\theta ,d\sigma )$ be a closed weighted strictly pseudoconvex CR $%
(2n+1)$-manifold. We say that the weighted CR Paneitz operator $P_{0}^{\phi
} $ is nonnegative if 
\begin{equation*}
\begin{array}{c}
\int_{M}\overline{f}P_{0}^{\phi }fd\sigma \geq 0%
\end{array}%
\end{equation*}%
for all smooth functions $f$.
\end{definition}

\begin{remark}
\label{r1} Let $(M,J,\theta ,d\sigma )$ be a closed weighted strictly
pseudoconvex CR $(2n+1)$-manifold of vanishing torsion and $\phi _{0}$
vanishes on $M$. It follows from (\ref{3a}) that the weighted Kohn Laplacian 
$\square _{b}^{\phi }+\sqrt{-2}nT$ and $\square _{b}^{\phi }-\sqrt{-2}nT$
commute and they are diagonalized simultaneously with 
\begin{equation*}
\begin{array}{c}
4P_{0}^{\phi }=\square _{b}^{\phi }\square _{b}^{\phi }+\overline{\square }%
_{b}^{\phi }\overline{\square }_{b}^{\phi }+4n^{2}T^{2}.%
\end{array}%
\end{equation*}%
Then the corresponding weighted CR Paneitz operator $P_{0}^{\phi }$ is
nonnegative.
\end{remark}

By using integrating by parts to the CR Bochner formula (\ref{11}) for $%
\square _{b}^{\phi }$ with respect to the given weighted volume measure $%
d\sigma $, we derive the following weighted CR Reilly formula.

\begin{theorem}
\label{Reilly'sformula} Let $(M,J,\theta ,d\sigma )$ be a closed weighted
strictly pseudoconvex CR $(2n+1)$-manifold. Then for any complex smooth
function $f$, we have%
\begin{equation}
\begin{array}{ll}
& \frac{n+1}{4n}\int_{M}\langle \square _{b}^{\phi }f,\square _{b}^{\phi
}f\rangle d\sigma  \\ 
= & \int_{M}\tsum_{\alpha ,\beta }|f_{\overline{\alpha }\overline{\beta }}-%
\frac{3}{4}f_{\overline{\alpha }}\phi _{\overline{\beta }}|^{2}d\sigma -%
\frac{1}{16}\int_{M}|\overline{\partial }_{b}f|^{2}[-12\Delta _{b}\phi +|%
\overline{\partial }_{b}\phi |^{2}]d\sigma  \\ 
& +\frac{1}{2n}\int_{M}\overline{f}P_{0}^{\phi }fd\sigma +\int_{M}[Ric+\frac{%
1}{4}Tor(\square _{b}^{\phi })]((\nabla _{b}f)_{\mathbb{C}},(\nabla _{b}f)_{%
\mathbb{C}})d\sigma  \\ 
& +\frac{1}{4n}\int_{M}\func{Re}\langle \square _{b}^{\phi }f,(2n-1)\langle 
\overline{\partial }_{b}f,\overline{\partial }_{b}\phi \rangle -(n-1)\langle
\partial _{b}f,\partial _{b}\phi \rangle \rangle d\sigma .%
\end{array}
\label{0}
\end{equation}
\end{theorem}

With its applications, we derive the first eigenvalue estimate of the
weighted Kohn Laplacian $\square _{b}^{\phi }$ and the CR\ Obata theorem in
a closed weighted strictly pseudoconvex CR $(2n+1)$-manifold.

\begin{theorem}
\label{TB} Let $(M,J,\theta ,d\sigma )$ be a closed weighted strictly
pseudoconvex CR $(2n+1)$-manifold with the nonnegative weighted CR Paneitz
operator $P_{0}^{\phi }$. Suppose that 
\begin{equation}
\begin{array}{c}
\lbrack Ric+\frac{1}{4}Tor(\square _{b}^{\phi })](Z,Z)\geq k\left\langle
Z,Z\right\rangle 
\end{array}
\label{B}
\end{equation}%
for all $Z\in T_{1,0}$ and a positive constant $k.$ Then the first
eigenvalue of the weighted Kohn Laplacian $\square _{b}^{\phi }$ satisfies
the lower bound 
\begin{equation}
\begin{array}{c}
\lambda _{1}\geq \frac{4n(k-l)}{n\omega +2n+2},%
\end{array}
\label{1}
\end{equation}%
where $\omega =\underset{M}{\mathrm{osc}}\phi =\underset{M}{\sup }\phi -%
\underset{M}{\inf }\phi ,$ and for nonnegative constant $l$ with $0\leq l<k$
such that%
\begin{equation}
\begin{array}{c}
-12\Delta _{b}\phi +|\overline{\partial }_{b}\phi |^{2}\leq 16l%
\end{array}
\label{1a}
\end{equation}%
on $M$. Moreover, if the equality (\ref{1}) holds, then $M$ is CR isometric
to a standard CR $(2n+1)$-sphere.
\end{theorem}

Furthermore, $P_{0}^{\phi }$ is nonnegative if the torsion is zero (i.e.
Sasakian) and $\phi _{0}$ vanishes (Lemma \ref{nonnegative}). Then we have
the following CR\ Obata theorem in a closed weighted Sasakian $(2n+1)$%
-manifold.

\begin{corollary}
\label{C}Let $(M,J,\theta ,d\sigma )$ be a closed weighted Sasakian $(2n+1)$%
-manifold with $\phi _{0}=0$. Suppose that 
\begin{equation*}
\begin{array}{c}
\lbrack Ric+\frac{1}{4}Tor(\square _{b}^{\phi })](Z,Z)\geq k\left\langle
Z,Z\right\rangle 
\end{array}%
\end{equation*}%
for all $Z\in T_{1,0}$ and a positive constant $k$. Then the first
eigenvalue of the weighted Kohn Laplacian $\square _{b}^{\phi }$ satisfies
the lower bound 
\begin{equation}
\begin{array}{c}
\lambda _{1}\geq \frac{4n(k-l)}{n\omega +2n+2},%
\end{array}
\label{2}
\end{equation}%
where $\omega =\underset{M}{\mathrm{osc}}\phi =\underset{M}{\sup }\phi -%
\underset{M}{\inf }\phi ,$ and for nonnegative constant $l$ with $0\leq l<k$
such that%
\begin{equation*}
\begin{array}{c}
-12\Delta _{b}\phi +|\overline{\partial }_{b}\phi |^{2}\leq 16l%
\end{array}%
\end{equation*}%
on $M$. Moreover, if the equality (\ref{2}) holds, then $M$ is CR isometric
to a standard CR $(2n+1)$-sphere.
\end{corollary}

Note that by comparing (\ref{A}), (\ref{B}) and (\ref{1a}), it has a plenty
of rooms for the choice of the weighted function $\phi $. For example, it is
the case by a small perturbation of the subhessian of the weighted function $%
\phi $.

We briefly describe the methods used in our proofs. In section $3$, we
introduce the weighted CR Paneitz operator $P_{0}^{\phi }$. In section $4$,
by using integrating by parts to the weighted CR Bochner formula (\ref{9}),
we can derive the CR version of weighted Reilly's formula. By applying such
formula, we are able to obtain the first eigenvalue estimate and weighted
Obata theorem as in section $5$ in a closed strictly pseudoconvex CR $(2n+1)$%
-manifold.

\section{Basic Notions in Pseudohermitian Geometry}

Let us give a brief introduction to pseudohermitian geometry (see \cite{l}
for more details). Let $(M,\xi )$ be a $(2n+1)$-dimensional, orientable,
contact manifold with contact structure $\xi ,\ \dim _{R}\xi =2n$. A CR
structure compatible with $\xi $ is an endomorphism $J:\xi \rightarrow \xi $
such that $J^{2}=-1$. We also assume that $J$ satisfies the following
integrability condition: If $X$ and $Y$ are in $\xi $, then so is $%
[JX,Y]+[X,JY]$ and $J([JX,Y]+[X,JY])=[JX,JY]-[X,Y]$. A CR structure $J$ can
extend to $\mathbb{C}\mathbf{\otimes }\xi $ and decomposes $\mathbb{C}%
\mathbf{\otimes }\xi $ into the direct sum of $T_{1,0}$ and $T_{0,1}$ which
are eigenspaces of $J$ with respect to $i$ and $-i$, respectively. A
pseudohermitian structure compatible with $\xi $ is a CR structure $J$
compatible with $\xi $ together with a choice of contact form $\theta $.
Such a choice determines a unique real vector field $T$ transverse to $\xi $%
, which is called the characteristic vector field of $\theta $, such that ${%
\theta }(T)=1$ and $\mathcal{L}_{T}{\theta }=0$ or $d{\theta }(T,{\cdot })=0$%
. Let $\left\{ T,Z_{\alpha },Z_{\bar{\alpha}}\right\} $ be a frame of $%
TM\otimes \mathbb{C}$, where $Z_{\alpha }$ is any local frame of $T_{1,0},\
Z_{\bar{\alpha}}=\overline{Z_{\alpha }}\in T_{0,1}$ and $T$ is the
characteristic vector field. Then $\left\{ \theta ,\theta ^{\alpha },\theta
^{\bar{\alpha}}\right\} $, which is the coframe dual to $\left\{ T,Z_{\alpha
},Z_{\bar{\alpha}}\right\} $, satisfies 
\begin{equation*}
\begin{array}{c}
d\theta =ih_{\alpha \bar{\beta}}\theta ^{\alpha }\wedge \theta ^{\bar{\beta}%
},%
\end{array}%
\end{equation*}%
for some positive definite hermitian matrix of functions $(h_{\alpha \bar{%
\beta}})$. Actually we can always choose $Z_{\alpha }$ such that $h_{\alpha 
\bar{\beta}}=\delta _{\alpha \beta };$ hence, throughout this paper, we
assume $h_{\alpha \bar{\beta}}=\delta _{\alpha \beta }.$

The Levi form $\left\langle \ ,\ \right\rangle $ is the Hermitian form on $%
T_{1,0}$ defined by%
\begin{equation*}
\begin{array}{c}
\left\langle Z,W\right\rangle =-i\left\langle d\theta ,Z\wedge \overline{W}%
\right\rangle .%
\end{array}%
\end{equation*}%
We can extend $\left\langle \ ,\ \right\rangle $ to $T_{0,1}$ by defining $%
\left\langle \overline{Z},\overline{W}\right\rangle =\overline{\left\langle
Z,W\right\rangle }$ for all $Z,W\in T_{1,0}$. The Levi form induces
naturally a Hermitian form on the dual bundle of $T_{1,0}$, also denoted by $%
\left\langle \ ,\ \right\rangle $, and hence on all the induced tensor
bundles.

The pseudohermitian connection of $(J,\theta )$ is the connection $\nabla $
on $TM\otimes \mathbb{C}$ (and extended to tensors) given in terms of a
local frame $Z_{\alpha }\in T_{1,0}$ by%
\begin{equation*}
\begin{array}{c}
\nabla Z_{\alpha }=\omega _{\alpha }{}^{\beta }\otimes Z_{\beta },\quad
\nabla Z_{\bar{\alpha}}=\omega _{\bar{\alpha}}{}^{\bar{\beta}}\otimes Z_{%
\bar{\beta}},\quad \nabla T=0,%
\end{array}%
\end{equation*}%
where $\omega _{\alpha }{}^{\beta }$ are the $1$-forms uniquely determined
by the following equations:%
\begin{equation*}
\begin{split}
d\theta ^{\beta }& =\theta ^{\alpha }\wedge \omega _{\alpha }{}^{\beta
}+\theta \wedge \tau ^{\beta }, \\
\tau _{\alpha }\wedge \theta ^{\alpha }& =0,\text{ \ }\omega _{\alpha
}{}^{\beta }+\omega _{\bar{\beta}}{}^{\bar{\alpha}}=0.
\end{split}%
\end{equation*}%
We can write $\tau _{\alpha }=A_{\alpha \beta }\theta ^{\beta }$ with $%
A_{\alpha \beta }=A_{\beta \alpha }$. The curvature of the Webster-Stanton
connection, expressed in terms of the coframe $\{\theta =\theta ^{0},\theta
^{\alpha },\theta ^{\bar{\alpha}}\}$, is 
\begin{equation*}
\begin{split}
\Pi _{\beta }{}^{\alpha }& =\overline{\Pi _{\bar{\beta}}{}^{\bar{\alpha}}}%
=d\omega _{\beta }{}^{\alpha }-\omega _{\beta }{}^{\gamma }\wedge \omega
_{\gamma }{}^{\alpha }, \\
\Pi _{0}{}^{\alpha }& =\Pi _{\alpha }{}^{0}=\Pi _{0}{}^{\bar{\beta}}=\Pi _{%
\bar{\beta}}{}^{0}=\Pi _{0}{}^{0}=0.
\end{split}%
\end{equation*}%
Webster showed that $\Pi _{\beta }{}^{\alpha }$ can be written 
\begin{equation*}
\begin{array}{c}
\Pi _{\beta }{}^{\alpha }=R_{\beta }{}^{\alpha }{}_{\rho \bar{\sigma}}\theta
^{\rho }\wedge \theta ^{\bar{\sigma}}+W_{\beta }{}^{\alpha }{}_{\rho }\theta
^{\rho }\wedge \theta -W^{\alpha }{}_{\beta \bar{\rho}}\theta ^{\bar{\rho}%
}\wedge \theta +i\theta _{\beta }\wedge \tau ^{\alpha }-i\tau _{\beta
}\wedge \theta ^{\alpha },%
\end{array}%
\end{equation*}%
where the coefficients satisfy 
\begin{equation*}
\begin{array}{c}
R_{\beta \bar{\alpha}\rho \bar{\sigma}}=\overline{R_{\alpha \bar{\beta}%
\sigma \bar{\rho}}}=R_{\bar{\alpha}\beta \bar{\sigma}\rho }=R_{\rho \bar{%
\alpha}\beta \bar{\sigma}},\ \ W_{\beta \bar{\alpha}\gamma }=W_{\gamma \bar{%
\alpha}\beta }.%
\end{array}%
\end{equation*}

We will denote components of covariant derivatives with indices preceded by
comma; thus write $A_{\alpha \beta ,\gamma }$. The indices $\{0,\alpha ,\bar{%
\alpha}\}$ indicate derivatives with respect to $\{T,Z_{\alpha },Z_{\bar{%
\alpha}}\}$. For derivatives of a function, we will often omit the comma,
for instance, $\varphi _{\alpha }=Z_{\alpha }\varphi ,\ \varphi _{\alpha 
\bar{\beta}}=Z_{\bar{\beta}}Z_{\alpha }\varphi -\omega _{\alpha }{}^{\gamma
}(Z_{\bar{\beta}})Z_{\gamma }\varphi ,\ \varphi _{0}=T\varphi $ for a
(smooth) function $\varphi $. The Cauchy-Riemann operator $\partial _{b}$ be
defined locally by $\partial _{b}\varphi =\varphi _{\alpha }\theta ^{\alpha
},$ and $\overline{\partial }_{b}$ be the conjugate of $\partial _{b}.$ If
we define $\partial _{b}f=f_{\alpha }\theta ^{\alpha }$ and $\bar{\partial}%
_{b}f=f_{\overline{\alpha }}\theta ^{\overline{\alpha }}$, then the formal
adjoint of $\partial _{b}$ on functions (with respect to the Levi form and
the volume form $\theta \wedge (d\theta )^{n}$) is $\partial _{b}^{\ast
}=-\delta _{b}$. For a function $\varphi $, the subgradient $\nabla _{b}$ is
defined locally by $\nabla _{b}\varphi =\varphi ^{\alpha }Z_{\alpha
}+\varphi ^{\overline{\alpha }}Z_{\bar{\alpha}}$. The sub-Laplacian $\Delta
_{b}$ and the Kohn Laplacian $\square _{b}$ on functions are defined by 
\begin{equation*}
\begin{array}{c}
\Delta _{b}\varphi =-(\varphi _{\alpha }{}^{\alpha }+\varphi _{\overline{%
\alpha }}{}^{\overline{\alpha }}),\text{ \ }\square _{b}\varphi =(\Delta
_{b}+\sqrt{-1}nT)\varphi =-2\varphi _{\overline{\alpha }}{}^{\overline{%
\alpha }}.%
\end{array}%
\end{equation*}%
The Webster-Ricci tensor and the torsion tensor on $T_{1,0}$ are defined by 
\begin{equation*}
\begin{array}{l}
Ric(X,Y)=R_{\alpha \bar{\beta}}X^{\alpha }Y^{\bar{\beta}}, \\ 
Tor(X,Y)=\sqrt{-1}\sum_{\alpha ,\beta }(A_{\bar{\alpha}\bar{\beta}}X^{\bar{%
\alpha}}Y^{\bar{\beta}}-A_{\alpha \beta }X^{\alpha }Y^{\beta }),%
\end{array}%
\end{equation*}%
where $X=X^{\alpha }Z_{\alpha },\ Y=Y^{\beta }Z_{\beta },\ R_{\alpha \bar{%
\beta}}=R_{\gamma }{}^{\gamma }{}_{\alpha \bar{\beta}}.$ The Webster scalar
curvature is $R=R_{\alpha }{}^{\alpha }=h^{\alpha \bar{\beta}}R_{\alpha \bar{%
\beta}}.$

\section{The weighted CR Paneitz operator}

Let $M$ be a closed strictly pseudoconvex CR $(2n+1)$-manifold with a
weighted volume measure $d\sigma =e^{-\phi (x)}\theta \wedge \left( d\theta
\right) ^{n}$ for a given smooth real function $\phi $. In this section, we
define the weighted CR Paneitz operator $P_{0}^{\phi }$. First we recall the
definition of the CR Paneitz operator $P_{0}.$

\begin{definition}
Let $(M,J,\theta )$ be a strictly pseudoconvex CR $(2n+1)$-manifold. We
define 
\begin{equation*}
\begin{array}{c}
Pf=\sum_{\gamma ,\beta =1}^{n}(f_{\overline{\gamma }\;\ \beta }^{\,\overline{%
\text{ }\gamma }}+\sqrt{-1}nA_{\beta \gamma }f^{\gamma })\theta ^{\beta
}=\sum_{\beta =1}^{n}(P_{\beta }f)\theta ^{\beta },%
\end{array}%
\end{equation*}%
which is an operator that characterizes CR-pluriharmonic functions. Here 
\begin{equation*}
\begin{array}{c}
P_{\beta }f=\sum_{\gamma =1}^{n}(f_{\overline{\gamma }\;\ \beta }^{\,\text{\ 
}\overline{\gamma }}+\sqrt{-1}nA_{\beta \gamma }f^{\gamma }),\text{ \ }\beta
=1,\cdots ,n,%
\end{array}%
\end{equation*}%
and $\overline{P}f=\sum_{\beta =1}^{n}\left( \overline{P}_{\beta }f\right)
\theta ^{\overline{\beta }}$, the conjugate of $P$. The CR Paneitz operator $%
P_{0}$ is defined by 
\begin{equation*}
\begin{array}{c}
P_{0}f=\delta _{b}(Pf)+\overline{\delta }_{b}(\overline{P}f),%
\end{array}%
\end{equation*}%
where $\delta _{b}$ is the divergence operator that takes $(1,0)$-forms to
functions by $\delta _{b}(\sigma _{\beta }\theta ^{\beta })=\sigma _{\beta
}^{\;\ \beta }$, and similarly, $\overline{\delta }_{b}(\sigma _{\overline{%
\beta }}\theta ^{\overline{\beta }})=\sigma _{\overline{\beta }}^{\;\ 
\overline{\beta }}$.
\end{definition}

We define the purely holomorphic second-order operator $Q$ by 
\begin{equation*}
\begin{array}{c}
Qf=2\sqrt{-1}(A^{\alpha \beta }f_{\alpha }),_{\beta }.%
\end{array}%
\end{equation*}%
It implies that $[T,\Delta _{b}]f=2\func{Im}Qf$ and observe from \cite{gl}
that 
\begin{equation}
\begin{array}{lll}
2P_{0}f & = & (\Delta _{b}^{2}+n^{2}T^{2})f-2n\mathrm{Re}Qf \\ 
& = & \square _{b}\overline{\square }_{b}f-2nQf\text{ }=\text{ }\overline{%
\square }_{b}\square _{b}f-2n\overline{Q}f.%
\end{array}
\label{3}
\end{equation}

With respect to the weighted volume measure $d\sigma =e^{-\phi (x)}\theta
\wedge \left( d\theta \right) ^{n},$ we define the purely holomorphic
second-order operator $Q^{\phi }$ by%
\begin{equation*}
\begin{array}{c}
Q^{\phi }f=Qf-2\sqrt{-1}(A^{\alpha \beta }f_{\alpha }\phi _{\beta })=2\sqrt{%
-1}e^{\phi }(e^{-\phi }A^{\alpha \beta }f_{\alpha }),_{\beta },%
\end{array}%
\end{equation*}%
and we have%
\begin{equation}
\begin{array}{c}
\lbrack T,\square _{b}^{\phi }]f=2\func{Im}Q^{\phi }f+2\left\langle 
\overline{\partial }_{b}f,\overline{\partial }_{b}\phi _{0}\right\rangle .%
\end{array}
\label{3a}
\end{equation}%
Now we can define the weighted CR Paneitz operator $P_{0}^{\phi }$ as follows%
\begin{equation*}
\begin{array}{c}
P_{0}^{\phi }f=e^{\phi }[\delta _{b}(e^{-\phi }P^{\phi }f)+\overline{\delta }%
_{b}(e^{-\phi }\overline{P}^{\phi }f)],%
\end{array}%
\end{equation*}%
here $P^{\phi }f=\sum_{\beta =1}^{n}(P_{\beta }^{\phi }f)\theta ^{\beta }$
with 
\begin{equation*}
\begin{array}{c}
P_{\beta }^{\phi }f=P_{\beta }f-\partial _{Z_{\beta }}\left\langle \partial
_{b}f,\partial _{b}\phi \right\rangle +n\sqrt{-1}f_{0}\phi _{\beta }%
\end{array}%
\end{equation*}%
and $\overline{P}^{\phi }f=\sum_{\beta =1}^{n}(\overline{P}_{\beta }^{\phi
}f)\theta ^{\overline{\beta }}$, the conjugate of $P^{\phi }$.

First we compare the relation between $P_{0}^{\phi }$ and $P_{0}.$

\begin{lemma}
Let $(M,J,\theta )$ be a closed strictly pseudoconvex CR $(2n+1)$-manifold.
We obtain%
\begin{equation}
\begin{array}{lll}
P_{0}^{\phi }f & = & P_{0}f-\langle Pf+\overline{P}f,d_{b}\phi \rangle +%
\frac{1}{2}[\overline{\square }_{b}^{\phi }\left\langle \partial
_{b}f,\partial _{b}\phi \right\rangle +\square _{b}^{\phi }\left\langle 
\overline{\partial }_{b}f,\overline{\partial }_{b}\phi \right\rangle ] \\ 
&  & +n\sqrt{{\small -1}}[\langle \overline{\partial }_{b}f_{0},\overline{%
\partial }_{b}\phi \rangle -\langle \partial _{b}f_{0},\partial _{b}\phi
\rangle ]-n^{2}f_{0}\phi _{0}.%
\end{array}
\label{4a}
\end{equation}
\end{lemma}

\begin{proof}
By the definition of $P_{0}^{\phi }$, we compute%
\begin{equation*}
\begin{array}{lll}
P_{0}^{\phi }f & = & \delta _{b}(P^{\phi }f)-\langle P^{\phi }f,\partial
_{b}\phi \rangle +\overline{\delta }_{b}(\overline{P}^{\phi }f)-\langle 
\overline{P}^{\phi }f,\overline{\partial }_{b}\phi \rangle \\ 
& = & (P_{\beta }f-\left\langle \partial _{b}f,\partial _{b}\phi
\right\rangle _{,\beta }+n\sqrt{{\small -1}}f_{0}\phi _{\beta })^{,\beta
}-\left\langle Pf-\partial _{b}\left\langle \partial _{b}f,\partial _{b}\phi
\right\rangle +n\sqrt{{\small -1}}f_{0}\partial _{b}\phi ,\partial _{b}\phi
\right\rangle \\ 
&  & +(P_{\overline{\beta }}f-\left\langle \overline{\partial }_{b}f,%
\overline{\partial }_{b}\phi \right\rangle _{,\overline{\beta }}-n\sqrt{%
{\small -1}}f_{0}\phi _{\overline{\beta }})^{,\overline{\beta }%
}-\left\langle \overline{P}f-\overline{\partial }_{b}\left\langle \overline{%
\partial }_{b}f,\overline{\partial }_{b}\phi \right\rangle -n\sqrt{{\small -1%
}}f_{0}\overline{\partial }_{b}\phi ,\overline{\partial }_{b}\phi
\right\rangle \\ 
& = & P_{0}f-\langle Pf+\overline{P}f,d_{b}\phi \rangle +\frac{1}{2}[%
\overline{\square }_{b}^{\phi }\left\langle \partial _{b}f,\partial _{b}\phi
\right\rangle +\square _{b}^{\phi }\left\langle \overline{\partial }_{b}f,%
\overline{\partial }_{b}\phi \right\rangle ] \\ 
&  & +n\sqrt{{\small -1}}[\langle \overline{\partial }_{b}f_{0},\overline{%
\partial }_{b}\phi \rangle -\langle \partial _{b}f_{0},\partial _{b}\phi
\rangle ]-n^{2}f_{0}\phi _{0}.%
\end{array}%
\end{equation*}
\end{proof}

Second we have the similar formula for $P_{0}^{\phi }$ like the expression (%
\ref{3}) for $P_{0}$.

\begin{lemma}
\label{lemma}Let $(M,J,\theta )$ be a closed strictly pseudoconvex CR $%
(2n+1) $-manifold. We have 
\begin{equation}
\begin{array}{l}
P_{0}^{\phi }f=\frac{1}{4}(\square _{b}^{\phi }\square _{b}^{\phi }+%
\overline{\square }_{b}^{\phi }\overline{\square }_{b}^{\phi
}+4n^{2}T^{2})f-n\func{Re}Q^{\phi }f-n^{2}f_{0}\phi _{0}.%
\end{array}
\label{4}
\end{equation}
\end{lemma}

\begin{proof}
By the straightforward calculation, we have%
\begin{equation*}
\begin{array}{lll}
\frac{1}{2}\square _{b}^{\phi }\square _{b}^{\phi }f & = & \frac{1}{2}%
\square _{b}^{\phi }(\square _{b}f+2\left\langle \overline{\partial }_{b}f,%
\overline{\partial }_{b}\phi \right\rangle ) \\ 
& = & \frac{1}{2}\square _{b}\square _{b}f+\langle \overline{\partial }%
_{b}\square _{b}f,\overline{\partial }_{b}\phi \rangle +\square _{b}^{\phi
}\left\langle \overline{\partial }_{b}f,\overline{\partial }_{b}\phi
\right\rangle \\ 
& = & \frac{1}{2}\square _{b}\square _{b}f-2\langle \overline{P}f,\overline{%
\partial }_{b}\phi \rangle +2n\sqrt{{\small -1}}(\langle \overline{\partial }%
_{b}f_{0},\overline{\partial }_{b}\phi \rangle -A^{\alpha \beta }f_{\alpha
}\phi _{\beta })+\square _{b}^{\phi }\left\langle \overline{\partial }_{b}f,%
\overline{\partial }_{b}\phi \right\rangle ,%
\end{array}%
\end{equation*}%
which implies%
\begin{equation*}
\begin{array}{ll}
& \frac{1}{4}[(\square _{b}^{\phi }\square _{b}^{\phi }+\overline{\square }%
_{b}^{\phi }\overline{\square }_{b}^{\phi })-(\square _{b}\square _{b}+%
\overline{\square }_{b}\overline{\square }_{b})]f \\ 
= & -\langle Pf+\overline{P}f,d_{b}\phi \rangle +\frac{1}{2}[\square
_{b}^{\phi }\left\langle \overline{\partial }_{b}f,\overline{\partial }%
_{b}\phi \right\rangle +\overline{\square }_{b}^{\phi }\left\langle \partial
_{b}f,\partial _{b}\phi \right\rangle ] \\ 
& +n\sqrt{{\small -1}}[\langle \overline{\partial }_{b}f_{0},\overline{%
\partial }_{b}\phi \rangle -\langle \partial _{b}f_{0},\partial _{b}\phi
\rangle +A^{\overline{\alpha }\overline{\beta }}f_{\overline{\alpha }}\phi _{%
\overline{\beta }}-A^{\alpha \beta }f_{\alpha }\phi _{\beta }].%
\end{array}%
\end{equation*}%
We then have%
\begin{equation*}
\begin{array}{ll}
& \frac{1}{4}(\square _{b}^{\phi }\square _{b}^{\phi }+\overline{\square }%
_{b}^{\phi }\overline{\square }_{b}^{\phi })f \\ 
= & P_{0}f-n^{2}f_{00}-\langle Pf+\overline{P}f,d_{b}\phi \rangle +\frac{1}{2%
}[\square _{b}^{\phi }\left\langle \overline{\partial }_{b}f,\overline{%
\partial }_{b}\phi \right\rangle +\overline{\square }_{b}^{\phi
}\left\langle \partial _{b}f,\partial _{b}\phi \right\rangle ] \\ 
& +n\sqrt{{\small -1}}[\langle \overline{\partial }_{b}f_{0},\overline{%
\partial }_{b}\phi \rangle -\langle \partial _{b}f_{0},\partial _{b}\phi
\rangle ]+n\func{Re}(Qf-2\sqrt{{\small -1}}A^{\alpha \beta }f_{\alpha }\phi
_{\beta }) \\ 
= & P_{0}^{\phi }f-n^{2}f_{00}+n\func{Re}Q^{\phi }f+n^{2}f_{0}\phi _{0},%
\end{array}%
\end{equation*}%
here we used the equation (\ref{3}) for $P_{0}$%
\begin{equation*}
\begin{array}{c}
\frac{1}{4}(\square _{b}\square _{b}+\overline{\square }_{b}\overline{%
\square }_{b})=\frac{1}{4}(\square _{b}\overline{\square }_{b}+\overline{%
\square }_{b}\square _{b})-n^{2}T^{2}=P_{0}-n^{2}T^{2}+n\func{Re}Q,%
\end{array}%
\end{equation*}%
the definition of $Q^{\phi }f$ and the relation (\ref{4a}) of $P_{0}^{\phi }$
and $P_{0}$.
\end{proof}

We explain why the weighted CR Paneitz operator $P_{0}^{\phi }$ to be
defined in this way. Recall the CR Bochner formula for $\Delta _{b},$ for a
real function $\varphi $, we have%
\begin{equation*}
\begin{array}{lll}
\frac{1}{2}\Delta _{b}|\nabla _{b}\varphi |^{2} & = & |(\nabla
^{H})^{2}\varphi |^{2}+\langle \nabla _{b}\varphi ,\nabla _{b}\Delta
_{b}\varphi \rangle +2\langle J\nabla _{b}\varphi ,\nabla _{b}\varphi
_{0}\rangle \\ 
&  & +[2Ric-(n-2)Tor]((\nabla _{b}\varphi )_{\mathbb{C}},(\nabla _{b}\varphi
)_{\mathbb{C}}),%
\end{array}%
\end{equation*}%
where $(\nabla _{b}\varphi )_{\mathbb{C}}=\varphi ^{\beta }Z_{\beta }$ is
the corresponding complex $(1,0)$-vector field of $\nabla _{b}\varphi $.
Comparing this formula with the Riemannian case, we have the extra term $%
\langle J\nabla _{b}\varphi ,\nabla _{b}\varphi _{0}\rangle $ is hard to
deal. From \cite{cc}, we can relate $\langle J\nabla _{b}\varphi ,\nabla
_{b}\varphi _{0}\rangle $ with $\langle \nabla _{b}\varphi ,\nabla
_{b}\Delta _{b}\varphi \rangle $ by%
\begin{equation*}
\begin{array}{c}
\langle J\nabla _{b}\varphi ,\nabla _{b}\varphi _{0}\rangle =\frac{1}{n}%
\langle \nabla _{b}\varphi ,\nabla _{b}\Delta _{b}\varphi \rangle -\frac{2}{n%
}\langle P\varphi +\overline{P}\varphi ,d_{b}\varphi \rangle -2Tor((\nabla
_{b}\varphi )_{\mathbb{C}},(\nabla _{b}\varphi )_{\mathbb{C}}),%
\end{array}%
\end{equation*}%
then by integral with respect to the volume measure $d\mu =\theta \wedge
\left( d\theta \right) ^{n}$ yields%
\begin{equation}
\begin{array}{c}
n^{2}\int_{M}\varphi _{0}^{2}d\mu =\int_{M}\left( \Delta _{b}\varphi \right)
^{2}d\mu -2\int_{M}\varphi P_{0}\varphi d\mu +2n\int_{M}Tor((\nabla
_{b}\varphi )_{\mathbb{C}},(\nabla _{b}\varphi )_{\mathbb{C}})d\mu .%
\end{array}
\label{5}
\end{equation}%
For our weighted case, by integral of equation (\ref{4}) times $\overline{f}$%
, for a complex function $f$, with respect to the weighted volume measure $%
d\sigma =e^{-\phi }\theta \wedge \left( d\theta \right) ^{n},$ one gets%
\begin{equation*}
\begin{array}{c}
n^{2}\int_{M}\langle f_{0},f_{0}\rangle d\sigma =\frac{1}{4}%
\int_{M}[|\square _{b}^{\phi }f|^{2}+|\overline{\square }_{b}^{\phi
}f|^{2}]d\sigma -\int_{M}\overline{f}P_{0}^{\phi }fd\sigma
+n\int_{M}Tor((\nabla _{b}f)_{\mathbb{C}},(\nabla _{b}f)_{\mathbb{C}%
})d\sigma .%
\end{array}%
\end{equation*}%
This integral have the same type as (\ref{5}).

In the end, we show that if the pseudohermitian torsion of $M$ is zero and $%
\phi _{0}$ vanishes, then the weighted CR Paneitz operator $P_{0}^{\phi }$
is nonnegative. Note that for any complex function $f=\varphi +\sqrt{{\small %
-1}}\psi ,$ we have%
\begin{equation*}
\begin{array}{c}
\int_{M}\overline{f}P_{0}^{\phi }fd\sigma =\int_{M}[\varphi P_{0}^{\phi
}\varphi +\psi P_{0}^{\phi }\psi ]d\sigma .%
\end{array}%
\end{equation*}

\begin{lemma}
\label{nonnegative}Suppose the pseudohermitian torsion of $M$ is zero and $%
\phi _{0}$ vanishes, then the weighted CR Paneitz operator $P_{0}^{\phi }$
is nonnegative. In particular, the weighted CR Paneitz operator $P_{0}^{\phi
}$ is nonnegative in a closed weighted Sasakian $(2n+1)$-manifold with $\phi
_{0}$ vanishes.
\end{lemma}

\begin{proof}
The zero pseudohermitian torsion and $\phi _{0}=0$ implies the weighted CR
Paneitz operator 
\begin{equation*}
\begin{array}{lll}
4P_{0}^{\phi } & = & \square _{b}^{\phi }\square _{b}^{\phi }+\overline{%
\square }_{b}^{\phi }\overline{\square }_{b}^{\phi }+4n^{2}T^{2} \\ 
& = & (\square _{b}^{\phi }-\sqrt{{\small -2}}nT)(\square _{b}^{\phi }+\sqrt{%
{\small -2}}nT)+(\overline{\square }_{b}^{\phi }+\sqrt{{\small -2}}nT)(%
\overline{\square }_{b}^{\phi }-\sqrt{{\small -2}}nT),%
\end{array}%
\end{equation*}%
and $\square _{b}^{\phi }-\sqrt{{\small -2}}nT$ and $\square _{b}^{\phi }+%
\sqrt{{\small -2}}nT$ are commute, since $[\square _{b}^{\phi },T]=0$, so
they are diagonalized simultaneously on the finite dimensional eigenspace of 
$\square _{b}^{\phi }+\sqrt{{\small -2}}nT$ with respect to any nonzero
eigenvalue. And we know that the eigenvalues of $\square _{b}^{\phi }-\sqrt{%
{\small -2}}nT$ (and thus of $\square _{b}^{\phi }+\sqrt{{\small -2}}nT$)
are all nonnegative, since for any real function $\varphi $ 
\begin{equation*}
\begin{array}{lll}
\int_{M}\langle \varphi ,(\square _{b}^{\phi }-\sqrt{{\small -2}}nT)\varphi
\rangle d\sigma & = & \int_{M}\varphi (\overline{\square }_{b}^{\phi
}\varphi +\sqrt{{\small -2}}n\varphi _{0})d\sigma \\ 
& = & \int_{M}[|\overline{\partial }_{b}\varphi |^{2}+\frac{1}{\sqrt{-2}}%
n\varphi ^{2}\phi _{0}]d\sigma \\ 
& = & \int_{M}|\overline{\partial }_{b}\varphi |^{2}d\sigma ,%
\end{array}%
\end{equation*}%
here we used the condition $\phi _{0}$ vanishes on $M$. Therefore, $%
P_{0}^{\phi }$ is nonnegative.
\end{proof}

\section{The weighted CR Reilly Formula}

In this section, we derive the weighted CR Reilly formula for the weighted
Kohn Laplacian on a closed strictly pseudoconvex CR $(2n+1)$-manifold $M$
with a given smooth real weighted function $\phi $.

First we recall the Bochner formula for the Kohn Laplacian (equation (2.8)
in \cite{ccy}).

\begin{proposition}
For any complex-valued function $f$, we have%
\begin{equation*}
\begin{array}{lll}
-\frac{1}{2}\square _{b}|\overline{\partial }_{b}f|^{2} & = & \tsum_{\alpha
,\beta }(f_{\overline{\alpha }\overline{\beta }}\overline{f}^{\overline{%
\alpha }\overline{\beta }}+f_{\overline{\alpha }\beta }\overline{f}^{%
\overline{\alpha }\beta })+Ric((\nabla _{b}f)_{\mathbb{C}},(\nabla _{b}f)_{%
\mathbb{C}}) \\ 
&  & -\frac{1}{2n}\langle \overline{\partial }_{b}f,\overline{\partial }%
_{b}\square _{b}f\rangle -\frac{n+1}{2n}\langle \overline{\partial }%
_{b}\square _{b}f,\overline{\partial }_{b}f\rangle  \\ 
&  & -\frac{1}{n}\langle \overline{P}f,\overline{\partial }_{b}f\rangle +%
\frac{n-1}{n}\langle P\overline{f},\partial _{b}\overline{f}\rangle ,%
\end{array}%
\end{equation*}%
where $(\nabla _{b}f)_{\mathbb{C}}=f^{\alpha }Z_{\alpha }$ is the
corresponding complex (1,0)-vector filed of $\nabla _{b}f.$
\end{proposition}

Now we have the following CR Bochner formula for $\square _{b}^{\phi }$.

\begin{lemma}
For any complex-valued function $f$, we have%
\begin{equation}
\begin{array}{lll}
-\frac{1}{2}\square _{b}^{\phi }|\overline{\partial }_{b}f|^{2} & = & 
\tsum_{\alpha ,\beta }[(f_{\overline{\alpha }\overline{\beta }}+f_{\overline{%
\alpha }}\phi _{\overline{\beta }})\overline{f}^{\overline{\alpha }\overline{%
\beta }}+(f_{\overline{\alpha }\beta }-f_{\overline{\alpha }}\phi _{\beta })%
\overline{f}^{\overline{\alpha }\beta }] \\ 
&  & +Ric(\square _{b}^{\phi })((\nabla _{b}f)_{\mathbb{C}},(\nabla _{b}f)_{%
\mathbb{C}})-\frac{1}{n}\langle \overline{P}^{\phi }f,\overline{\partial }%
_{b}f\rangle \\ 
&  & -\frac{1}{2n}\langle \overline{\partial }_{b}f,\overline{\partial }%
_{b}\square _{b}^{\phi }f\rangle -\frac{n+1}{2n}\langle \overline{\partial }%
_{b}\square _{b}^{\phi }f,\overline{\partial }_{b}f\rangle +\frac{n-1}{n}%
\langle P^{\phi }\overline{f},\partial _{b}\overline{f}\rangle \\ 
&  & -\sqrt{{\small -1}}[(n-1)\overline{f}_{0}\langle \overline{\partial }%
_{b}f,\overline{\partial }_{b}\phi \rangle +f_{0}\langle \overline{\partial }%
_{b}\phi ,\overline{\partial }_{b}f\rangle ].%
\end{array}
\label{9}
\end{equation}
\end{lemma}

The proof of the above formula follows from the definitions of $\square
_{b}^{\phi }$ and $P^{\phi }$, and the identity 
\begin{equation*}
\begin{array}{l}
\langle \overline{\partial }_{b}f,\overline{\partial }_{b}\langle \overline{%
\partial }_{b}f,\overline{\partial }_{b}\phi \rangle \rangle +\langle 
\overline{\partial }_{b}\langle \overline{\partial }_{b}f,\overline{\partial 
}_{b}\phi \rangle ,\overline{\partial }_{b}f\rangle -\langle \overline{%
\partial }_{b}|\overline{\partial }_{b}f|^{2},\overline{\partial }_{b}\phi
\rangle \\ 
=\overline{f}_{\alpha \beta }f^{\alpha }\phi ^{\beta }-\overline{f}_{\alpha 
\overline{\beta }}f^{\alpha }\phi ^{\overline{\beta }}+(\phi _{\alpha 
\overline{\beta }}+\phi _{\overline{\beta }\alpha })f^{\alpha }\overline{f}^{%
\overline{\beta }}.%
\end{array}%
\end{equation*}%
Note that by integrate both sides of the above equation with the weighted
volume measure $d\sigma $, one gets%
\begin{equation}
\begin{array}{ll}
& \int_{M}[\overline{f}_{\alpha \beta }f^{\alpha }\phi ^{\beta }-\overline{f}%
_{\beta \overline{\alpha }}f^{\beta }\phi ^{\overline{\alpha }}+(\phi
_{\alpha \overline{\beta }}+\phi _{\overline{\beta }\alpha })f^{\alpha }%
\overline{f}^{\overline{\beta }}]d\sigma \\ 
= & \frac{1}{2}\int_{M}[\langle \square _{b}^{\phi }f,\langle \overline{%
\partial }_{b}f,\overline{\partial }_{b}\phi \rangle \rangle +\langle
\langle \overline{\partial }_{b}f,\overline{\partial }_{b}\phi \rangle
,\square _{b}^{\phi }f\rangle -|\overline{\partial }_{b}f|^{2}\overline{%
\square }_{b}^{\phi }\phi ]d\sigma .%
\end{array}
\label{10}
\end{equation}%
Also note that 
\begin{equation*}
\begin{array}{c}
\langle \overline{\partial }_{b}f,\overline{\partial }_{b}\langle \overline{%
\partial }_{b}f,\overline{\partial }_{b}\phi \rangle \rangle +\langle 
\overline{\partial }_{b}\langle \overline{\partial }_{b}f,\overline{\partial 
}_{b}\phi \rangle ,\overline{\partial }_{b}f\rangle =f_{\overline{\alpha }%
\overline{\beta }}\overline{f}^{\overline{\alpha }}\phi ^{\overline{\beta }}+%
\overline{f}_{\alpha \beta }f^{\alpha }\phi ^{\beta }+(\phi _{\alpha 
\overline{\beta }}+\phi _{\overline{\beta }\alpha })f^{\alpha }\overline{f}^{%
\overline{\beta }}%
\end{array}%
\end{equation*}%
and%
\begin{equation*}
\begin{array}{c}
\langle \overline{\partial }_{b}\langle \partial _{b}f,\partial _{b}\phi
\rangle ,\overline{\partial }_{b}f\rangle -\langle \overline{\partial }%
_{b}\phi ,\overline{\partial }_{b}|\overline{\partial }_{b}f|^{2}\rangle
=\phi _{\overline{\alpha }\overline{\beta }}\overline{f}^{\overline{\alpha }%
}f^{\overline{\beta }}-\overline{f}_{\alpha \beta }f^{\alpha }\phi ^{\beta }+%
\sqrt{{\small -1}}f_{0}\langle \overline{\partial }_{b}\phi ,\overline{%
\partial }_{b}f\rangle ,%
\end{array}%
\end{equation*}%
then integral yields%
\begin{equation}
\begin{array}{ll}
& \int_{M}[f_{\overline{\alpha }\overline{\beta }}\overline{f}^{\overline{%
\alpha }}\phi ^{\overline{\beta }}+\overline{f}_{\alpha \beta }f^{\alpha
}\phi ^{\beta }+(\phi _{\alpha \overline{\beta }}+\phi _{\overline{\beta }%
\alpha })f^{\alpha }\overline{f}^{\overline{\beta }}]d\sigma \\ 
= & \frac{1}{2}\int_{M}[\langle \square _{b}^{\phi }f,\langle \overline{%
\partial }_{b}f,\overline{\partial }_{b}\phi \rangle \rangle +\langle
\langle \overline{\partial }_{b}f,\overline{\partial }_{b}\phi \rangle
,\square _{b}^{\phi }f\rangle ]d\sigma%
\end{array}
\label{10a}
\end{equation}%
and 
\begin{equation}
\begin{array}{ll}
& \sqrt{{\small -1}}\int_{M}f_{0}\langle \overline{\partial }_{b}\phi ,%
\overline{\partial }_{b}f\rangle d\sigma +\int_{M}[\phi _{\overline{\alpha }%
\overline{\beta }}\overline{f}^{\overline{\alpha }}f^{\overline{\beta }}-%
\overline{f}_{\alpha \beta }f^{\alpha }\phi ^{\beta }]d\sigma \\ 
= & \frac{1}{2}\int_{M}[\langle \langle \partial _{b}f,\partial _{b}\phi
\rangle ,\square _{b}^{\phi }f\rangle -|\overline{\partial }%
_{b}f|^{2}\square _{b}^{\phi }\phi ]d\sigma .%
\end{array}
\label{10b}
\end{equation}

\textbf{The Proof of Theorem}\textup{\textbf{\ \ref{Reilly'sformula}:}}

\begin{proof}
By integrating the CR Bochner formula (\ref{9}) for $\square _{b}^{\phi }$
with respect to the weighted volume measure $d\sigma $ and from (\ref{10}), (%
\ref{10a}) and (\ref{10b}), we have%
\begin{equation*}
\begin{array}{lll}
0 & = & \int_{M}(f_{\overline{\alpha }\overline{\beta }}\overline{f}^{%
\overline{\alpha }\overline{\beta }}+f_{\overline{\alpha }\beta }\overline{f}%
^{\overline{\alpha }\beta })d\sigma -\frac{n+2}{4n}\int_{M}\langle \square
_{b}^{\phi }f,\square _{b}^{\phi }f\rangle d\sigma +\frac{2-n}{2n}%
\int_{M}fP_{0}^{\phi }\overline{f}d\sigma \\ 
&  & +\int_{M}Ric((\nabla _{b}f)_{\mathbb{C}},(\nabla _{b}f)_{\mathbb{C}%
})d\sigma -\frac{1}{2}\int_{M}|\overline{\partial }_{b}f|^{2}[n\overline{%
\square }_{b}^{\phi }\phi -\square _{b}^{\phi }\phi ]d\sigma \\ 
&  & +\int_{M}[(n-1)(f_{\overline{\alpha }\overline{\beta }}\overline{f}^{%
\overline{\alpha }}\phi ^{\overline{\beta }}-\phi _{\alpha \beta }\overline{f%
}^{\alpha }f^{\beta })+\phi _{\overline{\alpha }\overline{\beta }}\overline{f%
}^{\overline{\alpha }}f^{\overline{\beta }}-\overline{f}_{\alpha \beta
}f^{\alpha }\phi ^{\beta }]d\sigma \\ 
&  & +\frac{1}{2}\int_{M}[(n-1)\langle \square _{b}^{\phi }f,\langle
\partial _{b}f,\partial _{b}\phi \rangle \rangle -\langle \langle \partial
_{b}f,\partial _{b}\phi \rangle ,\square _{b}^{\phi }f\rangle ]d\sigma \\ 
&  & +\frac{1}{2}\int_{M}[\langle \overline{\partial }_{b}f,\overline{P}%
^{\phi }f\rangle -\langle \overline{P}^{\phi }f,\overline{\partial }%
_{b}f\rangle +2\func{Re}\langle \square _{b}^{\phi }f,\langle \overline{%
\partial }_{b}f,\overline{\partial }_{b}\phi \rangle \rangle ]d\sigma .%
\end{array}%
\end{equation*}%
By taking the real part of above equation yields%
\begin{equation}
\begin{array}{ll}
& \frac{n+2}{4n}\int_{M}\langle \square _{b}^{\phi }f,\square _{b}^{\phi
}f\rangle d\sigma \\ 
= & \int_{M}\tsum_{\alpha ,\beta }[|f_{\overline{\alpha }\overline{\beta }}+%
\frac{n-2}{2}f_{\overline{\alpha }}\phi _{\overline{\beta }}|^{2}+f_{%
\overline{\alpha }\beta }\overline{f}^{\overline{\alpha }\beta }]d\sigma \\ 
& +\int_{M}[Ric-\frac{n-2}{2}Tor(\square _{b}^{\phi })]((\nabla _{b}f)_{%
\mathbb{C}},(\nabla _{b}f)_{\mathbb{C}})d\sigma \\ 
& +\frac{2-n}{2n}\int_{M}fP_{0}^{\phi }\overline{f}d\sigma -\int_{M}|%
\overline{\partial }_{b}f|^{2}[\frac{n-1}{2}\Delta _{b}\phi +\frac{n^{2}}{4}|%
\overline{\partial }_{b}\phi |^{2}]d\sigma \\ 
& +\frac{1}{2}\int_{M}\func{Re}\langle \square _{b}^{\phi }f,(n-2)\langle
\partial _{b}f,\partial _{b}\phi \rangle +2\langle \overline{\partial }_{b}f,%
\overline{\partial }_{b}\phi \rangle \rangle d\sigma .%
\end{array}
\label{11}
\end{equation}%
In the following we deal with the term $\int_{M}f_{\overline{\alpha }\beta }%
\overline{f}^{\overline{\alpha }\beta }d\sigma .$ First by decomposing $%
\overline{f}_{\alpha \overline{\beta }}$ into trace-free part $\overline{f}%
_{\alpha \overline{\beta }}-\frac{1}{n}\overline{f}_{\gamma }{}^{\gamma
}h_{\alpha \overline{\beta }}$ and trace part $\frac{1}{n}\overline{f}%
_{\gamma }{}^{\gamma }h_{\alpha \overline{\beta }},$ we have%
\begin{equation}
\begin{array}{lll}
f_{\overline{\alpha }\beta }\overline{f}^{\overline{\alpha }\beta } & = & 
\tsum_{\alpha ,\beta }|\overline{f}_{\alpha \overline{\beta }}-\frac{1}{n}%
\overline{f}_{\gamma }{}^{\gamma }h_{\alpha \overline{\beta }}|^{2}+\frac{1}{%
4n}\langle \square _{b}f,\square _{b}f\rangle \\ 
& = & \tsum_{\alpha ,\beta }|\overline{f}_{\alpha \overline{\beta }}-\frac{1%
}{n}\overline{f}_{\gamma }{}^{\gamma }h_{\alpha \overline{\beta }}|^{2}+%
\frac{1}{4n}\langle \square _{b}^{\phi }f,\square _{b}^{\phi }f\rangle \\ 
&  & +\frac{1}{n}[|\langle \overline{\partial }_{b}f,\overline{\partial }%
_{b}\phi \rangle |^{2}-\func{Re}\langle \square _{b}^{\phi }f,\langle 
\overline{\partial }_{b}f,\overline{\partial }_{b}\phi \rangle \rangle ].%
\end{array}
\label{12}
\end{equation}%
The divergence formula for the trace-free part of $\overline{f}_{\alpha 
\overline{\beta }}$: 
\begin{equation*}
\begin{array}{c}
B_{\alpha \overline{\beta }}\overline{f}=\overline{f}_{\alpha \overline{%
\beta }}-\frac{1}{n}\overline{f}_{\gamma }{}^{\gamma }h_{\alpha \overline{%
\beta }},%
\end{array}%
\end{equation*}%
is given by 
\begin{equation*}
\begin{array}{lll}
(B^{\alpha \overline{\beta }}f)(B_{\alpha \overline{\beta }}\overline{f}) & =
& (B^{\alpha \overline{\beta }}\overline{f})(B_{\alpha \overline{\beta }}f)=%
\overline{f}^{\alpha \overline{\beta }}(B_{\alpha \overline{\beta }}f)=(%
\overline{f}^{\alpha }B_{\alpha \overline{\beta }}f)^{,\overline{\beta }}-%
\frac{n-1}{n}\overline{f}^{\alpha }P_{\alpha }f \\ 
& = & (\overline{f}^{\alpha }B_{\alpha \overline{\beta }}f)^{,\overline{%
\beta }}-\frac{n-1}{n}(\overline{f}P_{\alpha }f)^{,\alpha }+\frac{n-1}{2n}%
\overline{f}(P_{0}f) \\ 
& = & e^{\phi }[(e^{-\phi }\overline{f}^{\alpha }B_{\alpha \overline{\beta }}%
\overline{f})^{,\overline{\beta }}-\frac{n-1}{n}(e^{-\phi }\overline{f}%
P_{\alpha }f)^{,\alpha }] \\ 
&  & +\frac{n-1}{2n}\overline{f}(P_{0}f-\langle Pf+\overline{P}f,d_{b}\phi
\rangle )+\frac{1}{2}B_{\alpha \overline{\beta }}f(\overline{f}^{\alpha
}\phi ^{\overline{\beta }}+\phi ^{\alpha }\overline{f}^{\overline{\beta }}),%
\end{array}%
\end{equation*}%
since $(B^{\alpha \overline{\beta }}f)(B_{\alpha \overline{\beta }}\overline{%
f})$ is invariant under conjugate as well as when replace $f$ by $\overline{f%
}$. Since%
\begin{equation*}
\begin{array}{ll}
& B_{\alpha \overline{\beta }}f(\overline{f}^{\alpha }\phi ^{\overline{\beta 
}}+\phi ^{\alpha }\overline{f}^{\overline{\beta }}) \\ 
= & \langle \overline{\partial }_{b}|\partial _{b}f|^{2},\overline{\partial }%
_{b}\phi \rangle +\langle \partial _{b}|\overline{\partial }%
_{b}f|^{2},\partial _{b}\phi \rangle -\overline{f}_{\overline{\alpha }%
\overline{\beta }}f^{\overline{\alpha }}\phi ^{\overline{\beta }}-\overline{f%
}_{\alpha \beta }f^{\alpha }\phi ^{\beta } \\ 
& +\frac{1}{2n}\langle \overline{\square }_{b}^{\phi }f,\langle \overline{%
\partial }_{b}f,\overline{\partial }_{b}\phi \rangle +\langle \partial
_{b}f,\partial _{b}\phi \rangle \rangle -\frac{2}{n}|\langle \partial
_{b}f,\partial _{b}\phi \rangle |^{2}+\sqrt{{\small -1}}f_{0}\langle
\partial _{b}\overline{f},\partial _{b}\phi \rangle%
\end{array}%
\end{equation*}%
and from the relation (\ref{4a}) of $P_{0}^{\phi }$ and $P_{0}$, we get%
\begin{equation*}
\begin{array}{ll}
& \sum_{\alpha ,\beta }|\overline{f}_{\alpha \overline{\beta }}-\frac{1}{n}%
\overline{f}_{\gamma }{}^{\gamma }h_{\alpha \overline{\beta }}|^{2}\text{ }=%
\text{ }(B^{\alpha \overline{\beta }}\overline{f})(B_{\alpha \overline{\beta 
}}f) \\ 
= & e^{\phi }[(e^{-\phi }\overline{f}^{\alpha }B_{\alpha \overline{\beta }%
}f)^{,\overline{\beta }}-\frac{n-1}{n}(e^{-\phi }\overline{f}P_{\alpha
}f)^{,\alpha }]+\frac{n-1}{2n}\overline{f}(P_{0}^{\phi }f) \\ 
& -\frac{n-1}{4n}\overline{f}[\overline{\square }_{b}^{\phi }\left\langle
\partial _{b}f,\partial _{b}\phi \right\rangle +\square _{b}^{\phi
}\left\langle \overline{\partial }_{b}f,\overline{\partial }_{b}\phi
\right\rangle -2n^{2}f_{0}\phi _{0}] \\ 
& +\frac{1}{2}\sqrt{{\small -1}}[f_{0}\langle \partial _{b}\overline{f}%
,\partial _{b}\phi \rangle -(n-1)\overline{f}(\langle \overline{\partial }%
_{b}f_{0},\overline{\partial }_{b}\phi \rangle -\langle \partial
_{b}f_{0},\partial _{b}\phi \rangle )] \\ 
& +\frac{1}{2}[\langle \overline{\partial }_{b}|\partial _{b}f|^{2},%
\overline{\partial }_{b}\phi \rangle +\langle \partial _{b}|\overline{%
\partial }_{b}f|^{2},\partial _{b}\phi \rangle -\overline{f}_{\overline{%
\alpha }\overline{\beta }}f^{\overline{\alpha }}\phi ^{\overline{\beta }}-%
\overline{f}_{\alpha \beta }f^{\alpha }\phi ^{\beta }] \\ 
& +\frac{1}{4n}\langle \overline{\square }_{b}^{\phi }f,\langle \overline{%
\partial }_{b}f,\overline{\partial }_{b}\phi \rangle +\langle \partial
_{b}f,\partial _{b}\phi \rangle \rangle -\frac{1}{n}|\langle \partial
_{b}f,\partial _{b}\phi \rangle |^{2}.%
\end{array}%
\end{equation*}%
By integrating both sides of the above equation, from (\ref{10b}) and apply
the equations 
\begin{equation*}
\begin{array}{c}
\int_{M}[\overline{f}\langle \overline{\partial }_{b}f_{0},\overline{%
\partial }_{b}\phi \rangle -\frac{1}{2}\overline{f}f_{0}\overline{\square }%
_{b}^{\phi }\phi +f_{0}\langle \overline{\partial }_{b}\overline{f},%
\overline{\partial }_{b}\phi \rangle ]d\sigma =0%
\end{array}%
\end{equation*}%
and%
\begin{equation*}
\begin{array}{ll}
& \sqrt{{\small -1}}\int_{M}f_{0}[\langle \overline{\partial }_{b}\phi ,%
\overline{\partial }_{b}f\rangle -\langle \partial _{b}\phi ,\partial
_{b}f\rangle ]d\sigma \\ 
= & \int_{M}[\overline{f}_{\overline{\alpha }\overline{\beta }}f^{\overline{%
\alpha }}\phi ^{\overline{\beta }}+\overline{f}_{\alpha \beta }f^{\alpha
}\phi ^{\beta }-\phi _{\overline{\alpha }\overline{\beta }}\overline{f}^{%
\overline{\alpha }}f^{\overline{\beta }}-\phi _{\alpha \beta }\overline{f}%
^{\alpha }f^{\beta }]d\sigma \\ 
& +\frac{1}{2}\int_{M}[\langle \langle \overline{\partial }_{b}f,\overline{%
\partial }_{b}\phi \rangle ,\overline{\square }_{b}^{\phi }f\rangle +\langle
\langle \partial _{b}f,\partial _{b}\phi \rangle ,\square _{b}^{\phi
}f\rangle ]d\sigma \\ 
& -\int_{M}|\overline{\partial }_{b}f|^{2}(\Delta _{b}\phi +2|\overline{%
\partial }_{b}\phi |^{2})d\sigma ,%
\end{array}%
\end{equation*}%
then take its real part and note that $(B^{\alpha \overline{\beta }}%
\overline{f})(B_{\alpha \overline{\beta }}f)$ is invariant when replace $f$
by $\overline{f}$, 
\begin{equation*}
\begin{array}{ll}
& \int_{M}\sum_{\alpha ,\beta }|\overline{f}_{\alpha \overline{\beta }}-%
\frac{1}{n}\overline{f}_{\gamma }{}^{\gamma }h_{\alpha \overline{\beta }%
}|^{2}d\sigma \\ 
= & \frac{n-1}{2n}\int_{M}f(P_{0}^{\phi }\overline{f})d\sigma +\frac{2n-1}{4}%
\int_{M}|\overline{\partial }_{b}f|^{2}[\Delta _{b}\phi +2|\overline{%
\partial }_{b}\phi |^{2}]d\sigma \\ 
& +\frac{1}{2}\int_{M}[(2n-3)\func{Re}\phi _{\overline{\alpha }\overline{%
\beta }}f^{\overline{\alpha }}\overline{f}^{\overline{\beta }}-(2n-1)\func{Re%
}f_{\overline{\alpha }\overline{\beta }}\overline{f}^{\overline{\alpha }%
}\phi ^{\overline{\beta }}-\frac{2}{n}|\langle \partial _{b}f,\partial
_{b}\phi \rangle |^{2}]d\sigma \\ 
& -\frac{1}{4n}\int_{M}\func{Re}\langle \square _{b}^{\phi }f,(2n-3)\langle 
\overline{\partial }_{b}f,\overline{\partial }_{b}\phi \rangle
+(2n^{2}-3n-1)\langle \partial _{b}f,\partial _{b}\phi \rangle \rangle
d\sigma .%
\end{array}%
\end{equation*}%
Therefore, from (\ref{12}), we final get%
\begin{equation}
\begin{array}{ll}
& \int_{M}f_{\overline{\alpha }\beta }\overline{f}^{\overline{\alpha }\beta
}d\sigma \\ 
= & \frac{1}{4n}\int_{M}\langle \square _{b}^{\phi }f,\square _{b}^{\phi
}f\rangle d\sigma +\frac{2n-1}{4}\int_{M}|\overline{\partial }%
_{b}f|^{2}[\Delta _{b}\phi +2|\overline{\partial }_{b}\phi |^{2}]d\sigma \\ 
& +\frac{n-1}{2n}\int_{M}f(P_{0}^{\phi }\overline{f})d\sigma +\int_{M}[\frac{%
2n-3}{2}\func{Re}\phi _{\overline{\alpha }\overline{\beta }}f^{\overline{%
\alpha }}\overline{f}^{\overline{\beta }}-\frac{2n-1}{2}\limfunc{Re}f_{%
\overline{\alpha }\overline{\beta }}\overline{f}^{\overline{\alpha }}\phi ^{%
\overline{\beta }}]d\sigma \\ 
& -\frac{1}{4n}\int_{M}\func{Re}\langle \square _{b}^{\phi }f,(2n+1)\langle 
\overline{\partial }_{b}f,\overline{\partial }_{b}\phi \rangle
+(2n^{2}-3n-1)\langle \partial _{b}f,\partial _{b}\phi \rangle \rangle
d\sigma .%
\end{array}
\label{14}
\end{equation}%
Substituting the equation (\ref{14}) into the right hand side of (\ref{11}),
we can get (\ref{0}). This completes the proof of the Theorem \ref%
{Reilly'sformula}.
\end{proof}

\section{First Eigenvalue Estimate and Weighted Obata Theorem}

In this section, by applying the weighted CR Reilly formula (\ref{0}), we
give the first eigenvalue estimate for the weighted Kohn Laplacian $\square
_{b}^{\phi }$ and derive the Obata-type theorem in a closed weighted
strictly pseudoconvex CR $(2n+1)$-manifold.

\textbf{The Proof of Theorem}\textup{\textbf{\ \ref{TB}:}}

\begin{proof}
Let $f$ be an eigenfunction of the weighted Kohn Laplacian $\square
_{b}^{\phi }$ with respect to the first nonzero eigenvalue $\lambda _{1}$,
i.e., $\square _{b}^{\phi }f=\lambda _{1}f$. Under the curvature condition 
\begin{equation*}
\begin{array}{c}
\lbrack Ric+\frac{1}{4}Tor(\square _{b}^{\phi })]((\nabla _{b}f)_{\mathbb{C}%
},(\nabla _{b}f)_{\mathbb{C}})\geq k|\overline{\partial }_{b}f|^{2}%
\end{array}%
\end{equation*}
for a positive constant $k,$ then the integral formula (\ref{0}) yields%
\begin{equation*}
\begin{array}{ll}
& \frac{n+1}{4n}\lambda _{1}^{2}\int_{M}|f|^{2}d\sigma  \\ 
\geq  & k\int_{M}|\overline{\partial }_{b}f|^{2}d\sigma +\frac{1}{2n}\mu
_{1}\int_{M}|f|^{2}d\sigma -l\int_{M}|\overline{\partial }_{b}f|^{2}d\sigma +%
\frac{1}{8}\lambda _{1}\int_{M}\langle \overline{\partial }_{b}|f|^{2},%
\overline{\partial }_{b}\phi \rangle d\sigma  \\ 
= & \frac{1}{2n}[n(k-l)\lambda _{1}+\mu _{1}]\int_{M}|f|^{2}d\sigma +\frac{1%
}{16}\lambda _{1}\int_{M}\phi (\square _{b}^{\phi }|f|^{2})d\sigma  \\ 
= & \frac{1}{2n}[n(k-l)\lambda _{1}+\mu _{1}]\int_{M}|f|^{2}d\sigma +\frac{1%
}{8}\lambda _{1}\int_{M}\phi \lbrack \lambda _{1}|f|^{2}-2|\overline{%
\partial }_{b}f|^{2}]d\sigma  \\ 
\geq  & \frac{1}{2n}[n(k-l)\lambda _{1}+\mu _{1}]\int_{M}|f|^{2}d\sigma +%
\frac{1}{8}\lambda _{1}\int_{M}[\lambda _{1}|f|^{2}(\underset{M}{\inf }\phi
)-2|\overline{\partial }_{b}f|^{2}(\underset{M}{\sup }\phi )]d\sigma  \\ 
= & \frac{1}{4n}[2n(k-l)\lambda _{1}+2\mu _{1}-\frac{n}{2}\omega \lambda
_{1}^{2}]\int_{M}|f|^{2}d\sigma ,%
\end{array}%
\end{equation*}%
here we assume $-12\Delta _{b}\phi +|\overline{\partial }_{b}\phi |^{2}\leq
16l$ for some nonnegative constant $l$ and 
\begin{equation*}
\begin{array}{c}
\int_{M}\overline{f}(P_{0}^{\phi }f)d\sigma \geq \mu
_{1}\int_{M}|f|^{2}d\sigma 
\end{array}%
\end{equation*}%
for the first eigenvalue $\mu _{1}$ of $P_{0}^{\phi },$ and let $\omega =%
\underset{M}{\mathrm{osc}}\phi =\underset{M}{\sup }\phi -\underset{M}{\inf }%
\phi .$ It implies\ that 
\begin{equation*}
\begin{array}{c}
(\frac{n}{2}\omega +n+1)\lambda _{1}^{2}-2n(k-l)\lambda _{1}-2\mu _{1}\geq 0,%
\end{array}%
\end{equation*}%
and thus%
\begin{equation*}
\begin{array}{c}
\lambda _{1}\geq \frac{2n(k-l)+2\sqrt{n^{2}\left( k-l\right) ^{2}+(n\omega
+2n+2)\mu _{1}}}{n\omega +2n+2}.%
\end{array}%
\end{equation*}%
If the weighted CR Paneitz operator $P_{0}^{\phi }$ is nonnegative, that is $%
\mu _{1}\geq 0$, then the first eigenvalue $\lambda _{1}$ will satisfies%
\begin{equation*}
\begin{array}{c}
\lambda _{1}\geq \frac{4n(k-l)}{n\omega +2n+2}.%
\end{array}%
\end{equation*}

Moreover, if the above inequality becomes equality, then $\phi =-\underset{M}%
{\mathrm{osc}}\phi $ and thus the weighted function $\phi $ will be
constant, we have $\square _{b}^{\phi }=\square _{b}$ and $P_{0}^{\phi
}=P_{0}$. In this case we can let $l=0,$ then the first eigenvalue of the
Kohn Laplacian $\square _{b}$ achieves the sharp lower bound 
\begin{equation*}
\begin{array}{c}
\lambda _{1}=\frac{2nk}{n+1},%
\end{array}%
\end{equation*}%
and it reduces to the original Obata-type theorem for the Kohn Laplacian $%
\square _{b}$ in a closed strictly pseudoconvex CR $(2n+1)$-manifold $%
(M,J,\theta ).$ It following from Li-Son-Wang \cite{lsw} that $M$\ is CR
isometric to a standard CR $(2n+1)$-sphere.
\end{proof}

\end{document}